\documentclass[11pt]{amsart}
\usepackage{amssymb}
\usepackage{lscape}

\theoremstyle{plain}
\newtheorem{theorem}{Theorem}[section]

\newtheorem{lemma}{Lemma}[section]
\newtheorem{cor}{Corollary}[section]

\theoremstyle{definition}
\newtheorem{rem}{Remark}[section]  %%JHE: remarks in roman

\newcommand{\Z}{{\mathbb Z}}
\newcommand{\Q}{{\mathbb Q}}
\newcommand{\C}{{\mathbb C}}

\newcommand{\kdots}{,\ldots ,}

\newcommand{\Zz}{{\mathbb Z}}
\newcommand{\Qq}{{\mathbb Q}}
\newcommand{\Cc}{{\mathbb C}}
\newcommand{\Rr}{{\mathbb R}}
\newcommand{\half}{\mbox{$\frac{1}{2}$}}
\newcommand{\medfrac}[2]{\mbox{\large$\textstyle{\frac{#1}{#2}}$\normalsize}}

\renewcommand{\subjclass}[1]{\thanks{\emph{2010 Mathematics Subject Classification:}~#1}}
\renewcommand{\keywords}[1]{\thanks{\emph{Keywords and Phrases:}~#1}}
\renewcommand{\date}{\thanks{\emph{Date:} \today}}

\begin{document}

\title[On nearly linear recurrence sequences]{On  nearly linear recurrence sequences}

\date
\thanks{Research supported in part by the OTKA grants NK104208, NK101680.}

\author[Shigeki Akiyama, Jan-Hendrik Evertse and Attila~Peth\H{o}]{Shigeki Akiyama, Jan-Hendrik Evertse and Attila~Peth\H{o}}
\address{Shigeki Akiyama\newline Institute of Mathematics, University of Tsukuba \newline
1-1-1 Tennodai, Tsukuba, Ibaraki, 350-0006 JAPAN}
\email{akiyama@math.tsukuba.ac.jp}
\address{Jan-Hendrik Evertse\newline Leiden University,
Mathematical Institute\newline
P.O. Box 9512\\
2300 RA Leiden\\
THE NETHERLANDS}
\email{evertse@math.leidenuniv.nl}
\address{Attila Peth\H{o}\newline Department of Computer Science, University of Debrecen\newline
H-4010 Debrecen, P.O. Box 12, HUNGARY}
\email{Petho.Attila@inf.unideb.hu}

\maketitle

\begin{center}
\emph{To Professor Robert Tichy on the occasion of his 60th birthday}
\end{center}
\vskip0.4cm

\begin{center}
\begin{minipage}{10cm}
\footnotesize
{\sc Abstract.}
A nearly linear recurrence sequence (nlrs) is a complex sequence $(a_n)$ with the property
that there exist complex numbers $A_0$,$\ldots$,
$A_{d-1}$ such that the sequence
$\big(a_{n+d}+A_{d-1}a_{n+d-1}+\cdots +A_0a_n\big)_{n=0}^{\infty}$ is bounded.
We give an asymptotic Binet-type formula for such sequences.
We compare $(a_n)$ with a natural linear recurrence sequence (lrs)
$(\tilde{a}_n)$ associated
with it and prove under certain assumptions
that the difference sequence $(a_n- \tilde{a}_n)$ tends to infinity.
We show that several finiteness results
for lrs, in particular the Skolem-Mahler-Lech
theorem and results on common terms of two lrs,
are not valid anymore for nlrs with
integer terms. Our main tool in these investigations is an observation
that lrs with transcendental terms may have large
fluctuations, quite different from lrs with algebraic terms.
On the other hand we show under certain hypotheses, that
though there may be infinitely many of them,
the common terms of two nlrs are very sparse.
The proof of this result combines our Binet-type formula with 
a Baker type estimate for logarithmic forms.
\end{minipage}
\end{center}
\vskip0.4cm

\section{Introduction}

This paper was motivated by the investigations on shift radix systems, defined in \cite{Akiyama-Borbely-Brunotte-Pethoe-Steiner:05}.
For real numbers $S_0,\dots,S_{d-1}$ and
initial values $s_0,\dots,s_{d-1}\in \Z$, the inequality
\begin{equation} \label{rekursion}
0\le s_{n+d} +S_{d-1}s_{n+d-1} + \dots + S_0s_n <1, \; n\ge 0,
\end{equation}
uniquely defines a sequence of integers $(s_n)$.
If $S_0,\dots,S_{d-1}\in \Z$ then $(s_n)$ is a linear recurrence sequence.
However, if some of the coefficients are non-integers, then we get sequences of
a different nature. In earlier papers
\cite{Akiyama-Brunotte-Pethoe-Steiner:06},
\cite{Akiyama-Pethoe-:12}, \cite{Lowenstein-Hatjispyros-Vivaldi}
the case $d=2, S_0=1$ and $|S_1|< 2$ was investigated, as a model of
discretized rotation in the plane.
In that case it is conjectured that the sequence $(s_n)$
is always periodic.

%Analyzing our first results it
%turned out that they are valid for much general assumptions.
In this paper, we largely generalize
the sequences given by shift radix systems.
Let $A_0,\dots,A_{d-1}\in \C$.  Let $(a_n)$ be a sequence of complex numbers and define the error sequence $(e_n)$ by the initial terms $e_0=\dots= e_{d-1} = 0$ and
by the equations
\begin{equation}\label{nlrs-definition}
e_{n+d} = a_{n+d} +A_{d-1}a_{n+d-1} + \dots + A_0a_n
\end{equation}
for $n\ge 0$. We call $(a_n)$ a {\it nearly linear recurrence sequence},
in shortcut nlrs, if for some choice of $d$ and $A_0\kdots A_{d-1}$,
the sequence $(|e_n|)$ is bounded.
The sequence $(s_n)$ from \eqref{rekursion}
is obviously an nlrs because in that case the terms of the error sequence lie in the interval $[0,1)$. An interesting number theoretical
example
is when $a_n$ lies in the integer ring $R$ of an imaginary quadratic
field and $e_n$ is chosen to be in a fundamental region of the lattice
associated with $R$, see \cite{PV}.

It is easily shown that for a given nlrs $(a_n)$, the set of polynomials
$B_tx^t+B_{t-1}x^{t-1}+\cdots +B_0$ with complex coefficients such that
the sequence $(\sum_{i=0}^t B_ia_{n+i})$ is bounded is an ideal of the
polynomial ring $\C [x]$, called the \emph{ideal of $(a_n)$}.
There is a unique, monic polynomial generating the ideal of $(a_n)$,
called the \emph{characteristic polynomial of $(a_n)$}.
This corresponds to the necessarily unique relation \eqref{nlrs-definition}
of minimal length for which $(e_n)$ is bounded.

We mention here that the characteristic polynomial of a linear recurrence
sequence (lrs) $(a_n)$ may be
different from the characteristic polynomial of $(a_n)$ when viewed as an nlrs.
For instance, the Fibonacci sequence $(a_n)$ given by $a_0=0$, $a_1=1$
and $a_{n+2}=a_{n+1}+a_n$ for $n\geq 0$ has characteristic polynomial
$x^2-x-1$ when viewed as an lrs, but characteristic polynomial
$x-\theta$ with $\theta =\half (1+\sqrt{5})$ when viewed as an nlrs,
since the sequence $(a_{n+1}-\theta a_n)$ is bounded.
Indeed we will see in Lemma \ref{charpol} (i) in
Section \ref{Binet_t_sec}, that the characteristic polynomial of an nlrs
does not have roots of modulus $<1$, and
the characteristic polynomial of an lrs
and that of the sequence viewed as an nlrs differ only by factors of the form
$x-\alpha$ with $|\alpha|<1$.

Let $(a_n)$ be an nlrs and
$$
P(x) = x^d + A_{d-1}x^{d-1}+\dots+ A_0
$$
its characteristic polynomial. Further, let $(e_n)$ be the error sequence
from \eqref{nlrs-definition}.
Define the generating function
$$
c(z) = \sum_{j=1}^{\infty} e_{d+j-1} z^{-j}.
$$
Since $(e_n)$ is bounded, $c(z)$ is convergent for all complex $z$
with $|z|>1$.
If, moreover, $(e_n)$ is a sequence of real numbers then
we have $c(\overline{z}) = \overline{c(z)} $ for all $z\in \C, |z|>1$, where $\overline{z}$ denotes the complex conjugate of $z$.

To $(a_n)$ we associate two lrs
$(\hat{a}_n)$ and $(\tilde{a}_n)$, as follows.
Let $(\hat{a}_n)$ denote the lrs having the initial terms
$\hat{a}_0=\dots = \hat{a}_{d-2}= 0, \hat{a}_{d-1} =1$
and satisfying the recursion
\begin{equation} \label{lrek}
\hat{a}_{n+d} +A_{d-1}\hat{a}_{n+d-1} + \dots + A_0\hat{a}_n =0.
\end{equation}
The lrs $(\tilde{a}_n)$ is defined by the same recursion (\ref{lrek})
with different initial terms $\tilde{a}_j = a_j\ (j=0, \dots, d-1)$.

For the distinct roots $\alpha_1,\dots, \alpha_h$ of $P(x)$ denote
by $m_1,\dots,m_h$ their respective multiplicities.
Although in this paper we are mainly interested in the separable case,
where all multiplicities are equal to $1$,
we recall the so called \emph{Binet formula}
\begin{equation} \label{Binet}
\hat{a}_n = \hat{g}_1(n) \alpha_1^n + \dots +\hat{g}_h(n) \alpha_h^n
\end{equation}
in general form.
Here the polynomials $\hat{g}_j(x)$ are of degree at most $m_j-1$ and with coefficients from the field $\Q(\alpha_1,\dots,\alpha_h)$ for $j=1,\dots,h$.
For $\tilde{a}_n$ we have a similar expression, with polynomials $\tilde{g}_j(x)$
instead of $\hat{g}_j(x)$. In the case that $P(x)$ is separable,
i.e., that all its roots are simple, the polynomials $\hat{g}_j(x)$,
$\tilde{g}_j(x)$ are just constants and we write $\hat{g}_j$, $\tilde{g}_j$
for them.
With these notions we will prove the following theorem,
the essential part of which is a Binet-type expression for nlrs.

\begin{theorem} \label{Binet_t}
Assume that the characteristic polynomial of the nlrs $(a_n)$ is separable and its zeros are ordered as
$$|\alpha_1|\ge \dots \ge |\alpha_{r_1}|> 1 = |\alpha_{r_1+1}| =\dots =|\alpha_{r_1+r_2}|,$$
where $r_1+r_2=d$.
Denote by $\tilde{g}_j,\hat{g}_j$ the (constant) coefficients of $\alpha_j^n, j=1,\dots,d$ in the expression \eqref{Binet} of $\tilde{a}_n$ and $\hat{a}_n$ respectively. Then
\begin{itemize}
\item[(i)] if $r_1>0$ and $r_2=0$ then
$$
a_n = (\tilde{g}_1+\hat{g}_1 c(\alpha_1)) \alpha_1^n + \dots + (\tilde{g}_{r_1}+\hat{g}_{r_1} c(\alpha_{r_1})) \alpha_{r_1}^n +O(1)
$$
and $\tilde{g}_i+\hat{g}_i c(\alpha_i)\not= 0$ for $i=1\kdots r_1$;
\vskip0.15cm
\item[(ii)] if $r_1>0$ and $r_2>0$ then
$$
a_n = (\tilde{g}_1+\hat{g}_1 c(\alpha_1)) \alpha_1^n + \dots + (\tilde{g}_{r_1}+\hat{g}_{r_1} c(\alpha_{r_1})) \alpha_{r_1}^n +O(n)
$$
and $\tilde{g}_i+\hat{g}_i c(\alpha_i)\not= 0$ for $i=1\kdots r_1$;
\vskip0.15cm
\item[(iii)] and if $r_1=0$ and $r_2>0$ then
$$
a_n = O(n).
$$
\end{itemize}
\end{theorem}

We prove this theorem in Section \ref{Binet_t_sec}.
It is easy to show that the converse of Theorem \ref{Binet_t} (i) is also true,
that is, if $a_n=\beta_1\alpha_1^n+\cdots +\beta_{r_1}\alpha_{r_1}^n+O(1)$
for certain constants $\beta_1\kdots \beta_{r_1}$ then $(a_n)$ is an nlrs.
We will present the simple proof in Section \ref{Binet_t_sec}.
Moreover, at the end of Section \ref{Binet_t_sec} we give examples
showing that the $O(n)$-term in Theorem \ref{Binet_t} (ii), (iii)
can not be improved.

In Section \ref{growth_sec} we prove first that the fluctuation of an lrs can be extremely large, then we analyze the distance between an nlrs and a naturally chosen lrs.
We also deduce some other consequences for nlrs.
First,
we show that if $(a_n)$ is an nlrs with separable characteristic
polynomial and $\alpha_1\kdots \alpha_{r_1}$
are as in Theorem \ref{Binet_t},
then the constants
$c_1\kdots c_{r_1}$ such that
$a_n=c_1\alpha_1^n+\cdots +c_{r_1}\alpha_{r_1}^n+O(n)$ are unique.
Second,
we prove that the analogue of the
Skolem-Lech-Mahler theorem, see e.g. \cite{Lech}, \cite{ST}, does not hold generally for nlrs with at least two dominating roots with equal absolute values.

In the last Section \ref{common} we investigate the common terms of nlrs,
i.e., the solutions $(k,m)$ in non-negative integers of the equation
\begin{equation}\label{a=b}
a_k=b_m
\end{equation}
for two nlrs $(a_n)$, $(b_n)$. We consider the case that the characteristic
polynomials of $(a_n)$, $(b_n)$ have multiplicatively independent, real algebraic
dominating roots of modulus larger than $1$.
For lrs $(a_n)$, $(b_n)$ we know that in that case \eqref{a=b} has only
finitely many solutions. We give an example, showing that for nlrs this is
no longer true.
On the other hand, we show that the solutions of \eqref{a=b} are very sparse.
More precisely,
we show that if $(k_1,m_1)$, $(k_2,m_2)$ are any two distinct solutions of \eqref{a=b}
with $\max (k_2,m_2)\geq\max (k_1,m_1)$, then in fact $\max (k_2,m_2)$
exceeds an exponential function of $\max (k_1,m_1)$.

\section{Proof of Theorem \ref{Binet_t}} \label{Binet_t_sec}

We start with a lemma which imposes some restrictions on the characteristic
polynomial of an nlrs.

\begin{lemma}\label{charpol}
Let $(a_n)$ be an nlrs with characteristic polynomial $P(x)$.
\begin{itemize}
\item[(i)] The roots of $P(x)$ all have modulus $\geq 1$.
\item[(ii)] Assume that $a_n=O(n)$ holds for all $n$. Then the roots
of $P(x)$ all have modulus equal to $1$.
\end{itemize}
\end{lemma}

\begin{proof}
Let $Q(x)=\sum_{i=0}^t Q_ix^i$ be in the ideal of $(a_n)$.
Let $\alpha$ be a zero of $Q(x)$ and write $Q(x)=(x-\alpha )R(x)$,
$R(x)=\sum_{i=0}^{t-1} R_ix^i$. Define the sequences
$(q_n)$, $(r_n)$ by
\begin{equation}
\label{DefQR}
q_{n+t}=\sum_{i=0}^t Q_ia_{n+i},\ \ r_{n+t-1}=\sum_{i=0}^{t-1} R_ia_{n+i}\ \
\mbox{for } n\ge 0.
\end{equation}
By assumption, the sequence $(q_n)$ is bounded. Putting $R_{-1}=R_t=0$, we
have
$Q_i=R_{i-1}-\alpha R_i$ for $i=0\kdots t$, hence
\begin{eqnarray}\label{sequence-relation}
q_{n+t}&=&\sum_{i=0}^t (R_{i-1}-\alpha R_i)a_{n+i}=
\sum_{i=1}^t R_{i-1}a_{n+i}-\alpha \sum_{i=0}^{t-1} R_ia_{n+i}
\\
\nonumber
&=&r_{n+t}-\alpha r_{n+t-1}.
\end{eqnarray}

{\bf (i)} We prove that if $|\alpha |<1$, then the sequence $(r_n)$ is also bounded,
i.e., $R(x)=Q(x)/(x-\alpha )$ is in the ideal of $(a_n)$. By repeatedly
applying this, we see that the ideal of $(a_n)$ contains a polynomial
all whose zeros have modulus $\ge 1$. In particular, the characteristic
polynomial of $(a_n)$, being a divisor of this polynomial, cannot have zeros
of modulus $<1$.

Let $C:=\max (|r_{t-1}|, |q_t|, |q_{t+1}|,\dots )$.
By \eqref{sequence-relation} we have
\[
|r_{n+t}|\leq C+|\alpha |\cdot |r_{n+t-1}|\ \ \mbox{for all } n\ge 0,
\]
implying
\[
|r_{n+t}|\leq C\cdot \big(1+|\alpha |+|\alpha |^2+\cdots +|\alpha |^{n+1}\big)\ \
\mbox{for all } n\ge 0.
\]
This shows that $|r_{n+t}|\leq C/(1-|\alpha |)$ for all $n\geq 0$, i.e.,
the sequence $(r_n)$ is bounded.

{\bf (ii)} We now prove that $(r_n)$ is bounded if $|\alpha |>1$ when $a_n=O(n)$. Then similarly
as above we can deduce that the characteristic polynomial of $(a_n)$ has no
roots of modulus $>1$. Assume that the sequence $(r_n)$ is not bounded.
Let $C:=\max (|q_t|, |q_{t+1}|,\ldots )$. There is $n_0$ such that
$|r_{n_0+t}|>1+C/(|\alpha |-1)$. By \eqref{sequence-relation} we have
$|r_{n+1+t}|\geq |\alpha | \cdot |r_{n+t}|-C$ for $n=n_0, n_0+1,\ldots$ and this
implies, by induction on $t$,
\begin{eqnarray*}
|r_{n+t}| &\ge&
|r_{n_0+t}|\cdot |\alpha |^{n-n_0}-C(1+|\alpha |+\cdots +|\alpha |^{n-n_0-1})
\\
&= &
|r_{n_0+t}|\cdot |\alpha |^{n-n_0}-C\medfrac{|\alpha |^{n-n_0}-1}{|\alpha |-1}.
\end{eqnarray*}
So $|r_{n+t}|\ge |\alpha |^{n-n_0}$ for $n\ge n_0$. This shows that
for $n\geq n_0+t$,
the sequence $(r_n)$ grows exponentially. On the other hand, from our assumption
$a_n=O(n)$ it follows that $r_n=O(n)$ from (\ref{DefQR}). Thus, our assumption that the
sequence $(r_n)$ is unbounded leads to a contradiction.
\end{proof}

We now turn to the proof of Theorem \ref{Binet_t}.
We keep the notation from the statement of that theorem.
We need a technical lemma,
originally given in the context of
shift radix systems, Lemma 2 of \cite{Petho:06}.

\begin{lemma}\label{osszeg}
We have
\begin{equation*}
a_n = \tilde{a}_n + \sum_{j=1}^{n-d+1} \hat{a}_{n-j} e_{d-1+j}.
\end{equation*}
\end{lemma}

\begin{proof}
By the definition of $(\hat{a}_n)$ and $(\tilde{a}_n)$, it is
clearly true for $n\le d-1$. Assume that it is true for
$n\le m+d-1$ with $m\ge 0$. Then
\begin{eqnarray*}
a_{m+d}-\tilde{a}_{m+d}
&=& e_{m+d}-\sum_{j=1}^d A_{d-j} (a_{m+d-j}-\tilde{a}_{m+d-j}) \\
&=& e_{m+d}-\sum_{j=1}^d A_{d-j} \sum_{k=1}^{m-j+1} \hat{a}_{m+d-j-k} e_{d-1+k}\\
&=& e_{m+d}-\sum_{k=1}^m e_{d-1+k} \sum_{j=1}^{\min(d, m+1-k)}
A_{d-j} \hat{a}_{m+d-j-k}.
\end{eqnarray*}
Using the definition of $(\hat{a}_n)$ we have
\begin{eqnarray*}
a_{m+d}-\tilde{a}_{m+d}&=&e_{m+d}-\sum_{k=1}^m e_{d-1+k}
(-\hat{a}_{m+d-k})\\
&=& \sum_{k=1}^{m+1} \hat{a}_{m+d-k} e_{d-1+k}
\end{eqnarray*}
which finishes the induction.
\end{proof}

\begin{proof}[Proof of Theorem \ref{Binet_t}]

Both sequences $(\tilde{a}_n), (\hat{a}_n)$ can be written in the form \eqref{Binet} with $\hat{g}_j(x)=\hat{g}_j$, $\tilde{g}_j(x)=\tilde{g}_j$
constants and $h=d$.
Lemma \ref{osszeg} implies
\begin{eqnarray*}
a_n &=& \sum_{i=1}^d \tilde{g}_i \alpha_i^n + \sum_{j=1}^{n-d+1} \sum_{i=1}^d \hat{g}_i e_{d-1+j} \alpha_i^{n-j} \\
&=& \sum_{i=1}^d \left(\tilde{g}_i \alpha_i^n + \hat{g}_i \sum_{j=1}^{n-d+1}e_{d-1+j} \alpha_i^{n-j} \right)\\
&=& \sum_{i=1}^d \alpha_i^n \left(\tilde{g}_i + \hat{g}_i \sum_{j=1}^{n-d+1}e_{d-1+j} \alpha_i^{-j} \right).
\end{eqnarray*}

If $r_1=0$ then the bases of all exponential terms lie in the closed unit disk. Thus all summands are bounded. Further the number of summands is bounded by $nd$. Thus we proved the theorem for $r_1=0$.

The function $c(z)$ is well defined outside the closed unit disk, among others for all $\alpha_1,\dots,\alpha_{r_1}$. Thus if $r_1>0$ then put
\begin{eqnarray*}
b_n &=& \sum_{i=1}^{r_1} \left(\tilde{g}_i + \hat{g}_i c(\alpha_i) \right)\alpha_i^n + \sum_{i=r_1+1}^d \left(\tilde{g}_i \alpha_i^n + \hat{g}_i \sum_{j=1}^{n-d+1}e_{d-1+j} \alpha_i^{n-j} \right)\\
&=& \sum_{i=1}^{r_1} \left(\tilde{g}_i + \hat{g}_i c(\alpha_i) \right)\alpha_i^n + O(r_2 n +1).
\end{eqnarray*}

Using this notation we obtain
\begin{eqnarray*}
b_n - a_n &=& \sum_{i=1}^{r_1} \hat{g}_i \alpha_i^n \left( c(\alpha_i) - \sum_{j=1}^{n-d+1}e_{d-1+j} \alpha_i^{-j} \right)\\
&=& \sum_{i=1}^{r_1} \hat{g}_i \alpha_i^n \left(\sum_{j=n-d+2}^{\infty}e_{d-1+j} \alpha_i^{-j} \right)\\
&=& O(|\alpha_1|^d) = O(1).
\end{eqnarray*}
From the above observations we immediately deduce (i)--(iii), except that
in (i),(ii) we still have to verify that
$\tilde{g}_i+\hat{g}_i c(\alpha_i)\not= 0$ for $i=1\kdots r_1$.
Let $I\subseteq\{ 1\kdots r_1\}$ be the set of indices $i$ with
$\beta_i:=\tilde{g}_i+\hat{g}_i c(\alpha_i)\not= 0$, and put
\[
c_n:=a_n-\sum_{i\in I} \beta_i\alpha_i^n.
\]
Then $(c_n)$ is an nlrs with $c_n=O(n)$ for all $n$. By
Lemma \ref{charpol} (ii), the characteristic polynomial $g(x)$
of $(c_n)$ has only roots of modulus $1$. In general, if $(u_n)$, $(v_n)$
are two nlrs with characteristic polynomials $P_1(x)$, $P_2(x)$,
then $u_n+v_n$ is an nlrs, and $P_1(x)P_2(x)$ is in the ideal of $(u_n+v_n)$.
In particular, $g(x)\prod_{i\in I} (x-\alpha_i)$ is in the ideal of $(a_n)$.
But since the characteristic polynomial of $(a_n)$ has zeros
$\alpha_1\kdots \alpha_{r_1}$, we must have $I=\{ 1\kdots r_1\}$.
\end{proof}

\begin{rem}
The assertion (iii) of Theorem \ref{Binet_t} remains true with
simple modifications for nlrs with inseparable characteristic polynomial,
but with %%SA: exponential ?
remainder term $O(n^{\kappa})$, where $\kappa$ is the maximum of
the multiplicities of the roots of the characteristic polynomial
of $(a_n)$.
%%JHE: obvious from the proof; for cases (i), (ii) I can
%%prove a similar Binet type formula but with a different proof.
%%When $\alpha_1$ is a multiple root,
%%the main term of $a_n$ has the form $f(n) \alpha_1^n$ with some function $f$.
%%JHE: I don't understand this; if r_1=0 there is no main term.
As in our Diophantine application we can not deal with this case,
we postpone the study of the inseparable case.
\end{rem}

\begin{rem}
The error term $O(n)$ in Theorem \ref{Binet_t} (ii), (iii) is best possible.
For instance, let $\alpha_1\kdots\alpha_{r_1}$, $\beta_1\kdots\beta_{r_1}$
be as above, let $\gamma$ be a non-zero complex number, and let $(a_n)$
be a sequence of complex numbers such that
$$
a_n = \beta_1 \alpha_1^n + \dots + \beta_{r_1} \alpha_{r_1}^n + \gamma n +O(1)
$$
holds for all $n\ge 1$. Then $(a_n)$ is an nlrs with characteristic polynomial
$(x-\alpha_1)\cdots (x-\alpha_{r_1})(x-1)$.
Similarly, if $a_n=\gamma n+O(1)$
holds for all $n\ge 1$ then $(a_n)$ is an nlrs with characteristic polynomial
$x-1$.
\end{rem}

\begin{rem}
It is easy to see that the converse of Theorem \ref{Binet_t} (i) is also true.
Indeed, let $(a_n)$ be a sequence of complex numbers such that
$$
a_n = \beta_1 \alpha_1^n + \dots + \beta_{r_1} \alpha_{r_1}^n + O(1)
$$
holds for all $n\ge 1$ with non-zero complex numbers $\alpha_1,\dots, \alpha_{r_1}, \beta_1,\dots, \beta_{r_1}$ satisfying $|\alpha_j|>1,\, j=1,\dots,r_1$.
Then $(a_n)$ is an nlrs with characteristic polynomial
$(x-\alpha_1)\cdots (x-\alpha_{r_1})$.
In general, we can show that if there exist non-zero polynomials $g_1,\dots,g_h$ and
complex numbers $\alpha_1,\dots,\alpha_h$ with $|\alpha_i|\ge 1$
such that
\begin{equation}\label{binet-converse}
a_n = g_1(n) \alpha_1^n + \dots +g_h(n) \alpha_h^n +O(1)
\end{equation}
then $(a_n)$ is an nlrs with characteristic polynomial
$$\prod_{i=1}^h (x-\alpha_i)^{1+\deg g_i}.$$
This is not true anymore if in \eqref{binet-converse} we replace
 the error term $O(1)$ by $O(n^{\kappa})$ with
a positive integer $\kappa$.
We give a counterexample in the simplest case when $a_n=O(n^{\kappa})$.
Take a sequence $(b_n)$ which is not eventually periodic, taking two values $\{-1,1\}$.
Then the sequence $(n^{\kappa} b_n)$ can not be an nlrs. Indeed, if there are $A_0,\dots A_{\nu-1}$ such that
$a_n=n^{\kappa} b_n$ satisfies
$$
a_{n+\nu}+A_{\nu-1} a_{n+\nu-1} + \dots + A_0 a_n= O(1),
$$
then by non-periodicity,
we can find two increasing sequences of integers $(N_j)$ and $(M_j)$
for $j=1,2,\dots$ such that
$$
(N_j+\nu)^{\kappa}b_{N_j+\nu}+A_{\nu-1} (N_j+\nu-1)^{\kappa} b_{N_j+\nu-1}+\dots +A_0 N_j^{\kappa} b_{N_j}=O(1),
$$
$$
(M_j+\nu)^{\kappa}b_{M_j+\nu}+A_{\nu-1} (M_j+\nu-1)^{\kappa} b_{M_j+\nu-1}+\dots +A_0 M_j^{\kappa} b_{M_j}=O(1)
$$
with $b_{N_j+k}=b_{M_j+k}$ for $k=0,\dots \nu-1$ and $b_{N_j+\nu}+ b_{M_j+\nu}=0$.
Dividing by $(N_j+\nu)^{\kappa}$ and $(M_j+\nu)^{\kappa}$ respectively and
taking the difference gives
an impossibility:
$$
2b_{N_j+\nu}=o(1)\ \ \mbox{as } j\to\infty .
$$
\end{rem}

\section{On the growth of nlrs} \label{growth_sec}

Combining Theorem \ref{Binet_t} with some Diophantine approximation
arguments we are able to prove lower and upper estimates for the growth of nlrs.
Specializing our results for lrs we get surprising facts in this case too. Moreover we can estimate the growth of the difference sequence $(a_n-\tilde{a}_n)$. We start with the analysis of a special case.

%Recall that complex numbers $\eta_1,\dots,\eta_r$ are
%multiplicatively independent if the inequality
%$\eta_1^{v_1}\cdots\eta_r^{v_r} = 1$ with
%integers $v_1,\dots,v_r$ holds only when $v_1=v_2=\dots=v_r=0$.
%Otherwise, $\eta_1,\dots,\eta_r$ are multiplicatively dependent.
%Let $\eta_1,\dots,\eta_r$ be complex numbers lying on the unit circle,
%but none of them is a root of unity. Let $\gamma_1,\dots,\gamma_r$ be non-zero complex %numbers.
The main result of this section is the following theorem.

\begin{theorem} \label{growth}
Assume that $r\ge 2$.

\begin{itemize}
\item[(i)] Let $\eta_1,\dots,\eta_r$ be any pairwise distinct
complex numbers lying on the unit circle and $\gamma_1,\dots,\gamma_r$
any non-zero complex numbers.
Then there exists a constant $d_1>0$ such that
    \begin{equation}\label{eqi}
    |\gamma_1\eta_1^n + \dots + \gamma_r\eta_r^n|> d_1
    \end{equation}
holds for infinitely many positive integers $n$.

\item[(ii)] Let $\eta_1,\dots,\eta_r$ be any pairwise distinct complex numbers lying on the unit circle such that at least one of the quotients $\eta_j/\eta_r$, $1\le j <r$ is not a root of unity
and $\gamma_1,\dots ,\gamma_{r-1}$ any non-zero complex numbers.
Then for all $d_2>1$ there exists $\gamma_r$ such that the inequality
    \begin{equation}\label{eqii}
    |\gamma_1\eta_1^n + \dots + \gamma_r\eta_r^n|< d_2^{-n}
    \end{equation}
holds for infinitely many positive integers $n$.
\end{itemize}
\end{theorem}

\noindent
%(JHE: Attila asked me to insert this remark; I don't know which is the optimal
%place for it so I have put it here. It is a simple application
%of the Borel-Cantelli Lemma, so perhaps it suffices to merely refer to that
%lemma and leave the proof to the reader)

\begin{rem}
In relation to (ii), we should remark here that as a consequence of the p-adic Subspace Theorem
of Schmidt and Schlickewei,
if $\gamma_1,\dots ,\gamma_r$, $\eta_1,\dots,\eta_{r}$ are all algebraic
and $|\eta_1|=\dots =|\eta_{r}|=1$, then for every $d_2>1$
there are only finitely many positive integers $n$ with \eqref{eqii},
see \cite{Sch_P} or \cite{Evertse}.

In fact, one can show that if $\gamma_1,\dots,\gamma_{r-1}$
are any non-zero complex numbers
and $\eta_1,\dots ,\eta_r$ any complex numbers on the unit circle,
then for almost all complex $\gamma_r$ in the sense of Lebesgue
measure, we have that for every $d_2>1$, inequality \eqref{eqii} holds
for only finitely many positive integers $n$. To see this, let
$S$ be the set of $\gamma_r\in\C$ for which there exists $d_2>1$
such that \eqref{eqii} holds for infinitely many $n$.
Then $S=\bigcup_{k=1}^{\infty} S_k$, where $S_k$ is the set of $\gamma_r\in\C$
such that \eqref{eqii} with $d_2=1+k^{-1}$ holds for infinitely $n$.
For fixed $n,k$, let $B_{n,k}$ be the set of $\gamma_r\in\C$ satisfying
\eqref{eqii} with $d_2=1+k^{-1}$. Then $B_{n,k}$ has Lebesgue measure
$\lambda (B_{n,k})=\pi(1+k^{-1})^{-2n}$, the measure of a ball in $\C$ of radius
$d_2^{-n}$.
Thus, $S_k$ is the set of $\gamma_r\in\C$ that are contained
in $B_{n,k}$ for infinitely many $n$. We have $\sum_{n=1}^{\infty} \lambda (B_{n,k})<\infty$ so by the Borel-Cantelli Lemma, $S_k$ has Lebesgue measure $0$.
But then, $S$ must have Lebesgue measure $0$.
\end{rem}

The proof of the second assertion of Theorem \ref{growth}
is based on the following Diophantine approximation result.

\begin{lemma} \label{dioph}
Let $\eta_1,\dots,\eta_r$ be any pairwise distinct complex numbers lying on the unit circle, at least one is not a root of unity. For every $d>0$ there are infinitely many $n$ such that
$|\eta_j^n-1|<d$ holds for $j=1,\dots,r$.

\end{lemma}

\begin{proof}
We use the inequality
$$
|e^z-1|=|z|\cdot|\sum_{n=1}^{\infty} z^{n-1}/n!|< |z|\cdot e^{|z|},
$$
which holds for all complex $z$.

Let $0<d<1$. Write $\eta_j=e^{2\pi i u_j}$ with
real numbers $u_j$ for $j=1,\dots,r$. Since by assumption
not all $\eta_j$ are roots of unity, at least one of the $u_j$ is irrational.
By Dirichlet's approximation theorem (see, e.g., \cite[Chap. XI, Thm. 200]{HardyWright},
there are infinitely many integers $n$ for which
there exist integers $m_j=m_j(n)$ such that $|nu_j-m_j|< d/c$
for $j=1,\dots ,r$, where $c=2\pi\cdot e^{2\pi}$.
%{\bf Put a reference to Dirichlet Theorem ?}
For these $n$,
$$
|\eta_j^n-1|=|e^{2\pi i (nu_j-m_j)}-1|< e^{2\pi |nu_j-m_j|}2\pi|nu_j-m_j|<d.
$$
\end{proof}

The second lemma holds under more general assumptions. Its proof was inspired by an idea we found in the Hungarian lecture notes of P. Tur\'an \cite{Turan} pp. 361--362.

\begin{lemma} \label{Turan}
Let $\eta_1,\dots,\eta_r$ pairwise different and lying on the unit circle and
$\gamma_1(x),\dots,\gamma_r(x)\in \C[x]$ non-zero.
Let $g(n) = \gamma_1(n) \eta_1^n + \dots + \gamma_r(n) \eta_r^n$ for $n\in \Z$. Assume that $|g(n)| \le G$ for all $n\ge n_0$.
Then for $j=1\kdots r$,
$\gamma_j(n)$ is a constant, say $\gamma_j$, satisfying $|\gamma_j| \le G$.
%%JHE: slight modification
\end{lemma}

\begin{proof}
For every real $\alpha\ge 0$,
complex number $\xi\neq 1$ with $|\xi|=1$ and integer $n_1\geq n_0$ we have
\begin{equation}\label{Turaniq}
\lim_{T\rightarrow \infty} \frac 1T \sum_{n=n_1}^{n_1+T-1}
\frac {\xi^n}{n^{\alpha}}
=0,
\end{equation}
which follows from Abel summation.

Let $n^{\alpha}$ be the highest power of $n$ occurring in
$\gamma_1(n),\ldots,\gamma_r(n)$. It may occur in various $\gamma_i(n)$.
Suppose for instance that it occurs in $\gamma_r(n)$ and that the
corresponding coefficient is $b$.
%Then ...
%Let $n^{\alpha}$ be the highest power of $n$ occurring
%in $\gamma_1(n),\dots,\gamma_r(n)$. Say this highest power occurs in
%$\gamma_r(n)$ and the corresponding coefficient is $b$.
Then for any $n_1\geq n_0$, %%JHE: n_0 is fixed, we let n_1 vary
\begin{eqnarray*}
\lim_{T\to\infty} \frac 1T \sum_{n=n_1}^{n_1+T-1} \frac{g(n)}{n^{\alpha}\eta_r^n} &=& \sum_{j=1}^{r-1}\lim_{T\to\infty} \frac 1T \sum_{n=n_1}^{n_1+T-1} \frac{\gamma_j(n)}{n^{\alpha}} \left(\frac{\eta_{j}}{\eta_r}\right)^n\\ 
&&\qquad\qquad +\  \sum_{j=1}^{r-1}\lim_{T\to\infty} \frac 1T \sum_{n=n_1}^{n_1+T-1} \frac{\gamma_r(n)}{n^{\alpha}}\ =\ b
\end{eqnarray*}
by \eqref{Turaniq}. So $|b|\leq G/n_1^{\alpha}$ for all $n_1\geq n_0$, implying $\alpha=0$. Hence
$\gamma_1(n),\dots,\gamma_r(n)$ are all constants, say $\gamma_1,\dots,\gamma_r$ respectively.

Let $1\le j\le r$. Then applying \eqref{Turaniq} with $\alpha=0$ we obtain
$$
\lim_{T\rightarrow \infty} \frac 1T \sum_{n=n_0}^{n_0+T-1}
\frac{g(n)}{\eta_{j}^n}=\lim_{T\rightarrow \infty} \frac 1T \sum_{n=n_0}^{n_0+T-1} \gamma_j = \gamma_j.
$$
On the other hand the modulus of the left hand side is clearly not greater than $G$.
\end{proof}

Now we are in the position to prove Theorem \ref{growth}.

\begin{proof}[Proof of Theorem \ref{growth}]

{\bf (i)}
Let $\Gamma := \max\{|\gamma_j|\;:\; j=1,\dots,r \}$ and let $0<d_1< \Gamma$. If \eqref{eqi} holds for only finitely many integers $n$ then there exists an $n_0$ such that
$$
|\gamma_1\eta_1^n + \dots + \gamma_r\eta_r^n|\le d_1
$$
holds for all $n>n_0$. Then by Lemma \ref{Turan} $|\gamma_j|\le d_1<\Gamma$ for all $j=1,\dots,r$, which is a contradiction.
\\[1ex]

{\bf (ii)}
Dividing $\gamma_1\eta_1^n + \dots + \gamma_r\eta_r^n$ by $\eta_r^n$, we see that without loss of generality
we may assume that $\eta_r=1$ and at least one of $\eta_1,\dots,\eta_{r-1}$ is
not a root of unity.
We take any non-zero $\gamma_1,\dots,\gamma_{r-1}$ and construct $\gamma_r$.

Let $u_n:=\gamma_1\eta_1^n + \dots + \gamma_{r-1}\eta_{r-1}^n$.
We construct a sequence $(n_k)$. Let $n_1:=1$.
For $k\ge 1$, given $n_k$, choose $n_{k+1}>n_k$ such that
$$
|\eta_j^{n_{k+1}-n_k}-1|<(2d_2)^{-n_k}/|rB|,\; j=1,\dots,r-1,
$$
where $B>\max (|\gamma_1|,\dots ,|\gamma_{r-1}|)$. This is possible by Lemma \ref{dioph}. Then
\begin{eqnarray*}
|u_{n_{k+1}}-u_{n_k}|&\le&\sum_{j=1}^{r-1}|\gamma_j\eta_j^{n_k}|\cdot|\eta_j^{n_{k+1}-n_k}-1|\\
                       &<& (2d_2)^{-n_k}
\end{eqnarray*}
for $k=1,2,\ldots$

Now let
$$
\gamma_r:=-u_{n_1}-\sum_{k\ge 1}(u_{n_{k+1}}-u_{n_k}).
$$
This is a convergent series, and for $l\ge 1$,
\begin{eqnarray*}
|\gamma_1\eta_1^{n_l}+\dots+\gamma_{r-1}\eta_{r-1}^{n_l}+\gamma_r|&=&|u_{n_l}+\gamma_r|\\
&=&|\sum_{k\ge l}(u_{n_{k+1}}-u_{n_k})|\\
&<&(2d_2)^{-n_l}+(2d_2)^{-n_{l+1}}+\dots\\
              &\le& \frac{2d_2}{2d_2-1} (2d_2)^{-n_l}< d_2^{-n_l},
\end{eqnarray*}
completing the proof.
\end{proof}

Theorem \ref{growth} implies that general
linear recurrence sequences may have surprisingly big fluctuation.

\begin{cor}
  Let $r\ge 2$ be an integer and $h>1$ be a real number.
There exists a
lrs $u_n$ of degree $r$ such that $u_n\not= 0$ for all $n$,
$|u_n| \gg h^n$ for infinitely many $n$ and $|u_n| \ll h^{-n}$ for infinitely many $n$.
\end{cor}

Note that the non zero assumption expels trivial `degenerate' sequences
like $u_n=h^{2n}(1+\gamma^n+\gamma^{2n}+\dots +\gamma^{(r-1)n})$
for a primitive $r$-th root of unity $\gamma$.

\begin{proof}
Take distinct algebraic numbers
$\eta_1,\dots,\eta_{r-1}$ that lie on the unit circle and
are not roots of unity
and set $\eta_r=1$.
%As there are Salem numbers of arbitrary degree such algebraic numbers exist.
%For example we make take $\eta_i=(1-\sqrt{-p_i})/(1+\sqrt{-p_i})$
%where $p_i$ is the $i$-th prime.
Let $D\ge h$ be an integer
and put $\alpha_j=D\eta_j$ for $j=1,\dots,r$.
Finally let $\gamma_j, j=1,\dots,r-1$ be non-zero integers.
Taking $d_2=D^2$
there exists by Theorem \ref{growth} (ii) a complex number $\gamma_r$ such that
  $$
  |\gamma_1 \eta_1^n+\dots+ \gamma_r \eta_r^n| \ll D^{-2n}
  $$
  holds for infinitely many $n$. Let $u_n = \gamma_1 \alpha_1^n+\dots+ \gamma_r \alpha_r^n$. Then $u_n$ satisfies a linear recursive recursion, for which we have
  $$
  |u_n| = D^n \cdot |\gamma_1 \eta_1^n+\dots+ \gamma_r \eta_r^n| \ll D^{-n} \ll h^{-n}
  $$
  for infinitely many $n$.

  We claim that $\gamma_r$ is transcendental. Indeed, assume that $\gamma_r$ is algebraic. Then by \cite{Sch_P} for every $\varepsilon >0$
  we have $|u_n|\gg D^{n(1-\varepsilon)}$ for sufficiently large
  $n$, which is a contradiction.
Thus $\gamma_r$ is transcendental and as $\eta_1,\dots, \eta_{r}$ and $\gamma_1,\dots,\gamma_{r-1}$ are algebraic, we have $u_n\neq 0$ for all $n$.
By using Theorem \ref{growth} (i),
$|u_n| \gg D^n$ for infinitely many $n$. 
\end{proof}

We deduce some consequences for nlrs.

\begin{cor}\label{Diff4}
Let $(a_n)$ be an nlrs with separable characteristic polynomial
and assume that $\alpha_1\kdots \alpha_{r_1}$ are
its zeros of modulus $>1$ with $r_1\ge 1$.
Then there are unique
complex numbers $\beta_1\kdots\beta_{r_1}$ such that
\[
a_n=\beta_1\alpha_1^n+\cdots +\beta_{r_1}\alpha_{r_1}^n+O(n)
\]
holds for all $n\ge 1$.
\end{cor}

\begin{proof} Such $\beta_1\kdots\beta_{r_1}$ exist by Theorem
\ref{Binet_t}. Suppose there is also a tuple
of complex numbers
$(\gamma_1\kdots\gamma_{r_1})\not=(\beta_1\kdots\beta_{r_1})$
such that $a_n=\sum_{i=1}^{r_1}\gamma_i\alpha_i^n+O(n)$ for all $n$.
Let $k$ be an index $i$ for which $\gamma_i\not=\beta_i$
and $|\alpha_i|$ is maximal. Then
\[
\sum_{i=1}^{r_1}(\gamma_i-\beta_i)(\alpha_i/\alpha_k)^n=
O(n\cdot |\alpha_k|^{-n})\ \
\mbox{as } n\to\infty.
\]
But this clearly contradicts Theorem \ref{growth} (i).
\end{proof}

Recall that the Skolem-Mahler Lech theorem,
see e.g. \cite{Lech} or \cite{ST}, asserts that if $(a_n)$ is a lrs, then the set of $n$
with $a_n=0$ is either finite or contains an infinite arithmetic progression.
We show that there is no analogue for nlrs.

\begin{cor} \label{nlrs_common}
There exists an nlrs with integer terms $(a_n)$ such that\\
$\limsup_{n\to\infty} |a_n|= \infty$, but $a_n=0$ for infinitely many $n$
and the set of $n$ with $a_n=0$ does not contain an
infinite arithmetic progression.
\end{cor}

\begin{proof}
Let
$\alpha_1\kdots\alpha_r$ ($r\ge 2$) be complex numbers such that
\[
|\alpha_1|=\cdots =|\alpha_r|>1,
\]
none of the quotients $\alpha_i/\alpha_j$ ($1\leq i<j\leq r)$ is a root of unity,
and $\overline{\alpha_i}\in\{\alpha_1\kdots\alpha_r\}$ for $i=1\kdots r$.
Choose non-zero $\gamma_1\kdots\gamma_{r-1}\in\Cc$.
Let $C>1$. By Theorem \ref{growth} (ii)
there exists $\gamma_r\in\Cc$ such that
\[
|\gamma_1\alpha_1^n+\cdots +\gamma_r\alpha_r^n|<C^{-n}
\]
for infinitely many $n$. Let $t_n$ denote the real part of
$\sum_{i=1}^r \gamma_i\alpha_i^n$ for all $n\geq 0$ or, in case this is identically $0$,
$t_n=\frac{1}{2\sqrt{-1}}\cdot \sum_{i=1}^r \gamma_i\alpha_i^n$ for all $n$.
Then $t_n$ is real for all $n$ and
$|t_n|<C^{-n}$ for infinitely many $n$,
and by our assumption on the $\alpha_i$-s,
there are $\delta_1\kdots\delta_r$, not all $0$
such that $t_n=\sum_{i=1}^r \delta_i\alpha_i^n$ for all $n$.  Now we take
$a_n:=\lfloor t_n\rceil$, where $\lfloor x \rceil := [x+1/2]$ for $x\in\Rr$.
Then clearly, $(a_n)$ is an nlrs in $\Z$ and $a_n=0$ for infinitely many $n$.

It remains to prove that the set of $n$ with $a_n=0$ does not contain
an arithmetic progression. Consider the arithmetic progression
$u,u+v,u+2v,\ldots$. By Theorem \ref{growth} (i), there are
a constant $c>0$ and infinitely many integers
$m$ such that
\[
|t_{u+mv}|=\left|(\delta_1\alpha_1^u)(\alpha_1^v)^m+\cdots +(\delta_r\alpha_r^u)(\alpha_r^v)^m\right|>c|\alpha_1^v|^m.
\]
This implies that $|a_{u+mv}|>c'|\alpha_1^v|^m$ for infinitely many $m$,
where $0<c'<c$, so in particular, $a_{u+mv}\not= 0$
for infinitely many $m$. This shows at the same time that
$\limsup_{n\to\infty} |a_n|=\infty$.
\end{proof}

In the next corollaries, we compare the nlrs $(a_n)$ and its corresponding lrs analogue $(\tilde{a}_n)$. Although $a_n=\tilde{a}_n$ for $0\le n<d$,
we can show under a mild condition that the difference
$a_n-\tilde{a}_n$ can not be bounded. %%JHE slight modification

\begin{cor}
\label{Diff}
Under the same assumptions as in Theorem \ref{Binet_t} set
$$
R= \{\alpha_i\ | \ i=1,\dots,r_1 \text{ and } c(\alpha_i)\neq 0\}.
$$
Assume that $R\not=\emptyset$. %%JHE: otherwise false
If among the elements of $R$ there is exactly one
of maximum modulus, then $\lim_{n\to\infty}|a_n-\tilde{a}_n|=\infty$, otherwise
%  \item[(ii)] if among
%  $\{\alpha_i\ | \ i=1,\dots,r_1 \text{ and } c(\alpha_i)\neq 0\}$
%there are exactly $s$ elements of maximum modulus, where $1<s\leq r_1$,
%then
$$
\limsup_{n \to \infty}\, |a_n-\tilde{a}_n| = \infty.
$$
%%JHE: notion of divergence unclear; absolute value of a_n-\tilde{a_n} taken
\end{cor}

\begin{proof}
Observe that the coefficients $\hat{g}_j$ in \eqref{Binet} are all non-zero. Indeed,
otherwise ${\hat{a}_i}$ would be a lrs of order
less than $d$ and hence identically $0$, which it isn't.

$(i)$ Let $\alpha_i$ be the element of $R$ of
maximum modulus. Then
by Theorem \ref{Binet_t}, we have
$$
a_n-\tilde{a}_n = \hat{g}_ic(\alpha_i)\alpha_i^n + o(|\alpha_i|^n).
$$

$(ii)$ Let $\alpha_{i_1},\dots ,\alpha_{i_s}$ be the elements
of $R$ of maximum modulus.
As in case $(i)$ we have
\begin{eqnarray*}
a_n-\tilde{a}_n &=& \hat{g}_{i_1}c(\alpha_{i_1})\alpha_{i_1}^n + \dots + \hat{g}_{i_s}c(\alpha_{i_s})\alpha_{i_s}^n + o(|\alpha_{i_s}|^n)\\
 &=& d(n) |\alpha_{i_s}|^n + o(|\alpha_{i_s}|^n),
\end{eqnarray*}
where
$$
d(n) = \hat{g}_{i_1}c(\alpha_{i_1})\left(\frac{\alpha_{i_1}}{|\alpha_{i_s}|} \right)^n + \dots + \hat{g}_{i_s}c(\alpha_{i_{s}}) \left(\frac{\alpha_{i_{s}}}{|\alpha_{i_s}|} \right)^n.
$$
%%JHE: d depends on n
As the assumptions of Theorem \ref{growth} (i) hold, we can ensure  that
$|d(n)|>d_0>0$ for infinitely many $n$, and the proof is complete.
\end{proof}

In the next corollary we need stronger assumptions on the nlrs.
\begin{cor}
\label{Diff2}
Assume that $A_0,\dots,A_{d-2}$ are real, the terms of the nlrs $(a_n)$ are integers and $e_n\ge 0$ for all $n$, where $(e_n)$ denotes the corresponding error sequence. Further assume that the characteristic polynomial
has a single root of maximum modulus, which is real,
greater than one and not an algebraic integer.
Then $\lim_{n\to\infty} |a_n-\tilde{a}_n|=\infty$.
\end{cor}
%%JHE: I believe in Corollary 3.3 it has to be assumed that the nlrs has
%%a unique root of maximum modulus; otherwise
%%the first part of Corollary 3.2 is not applicable. Under the assumption
%%\alpha_1>1 as far as I can see one can prove only \limsup =\infty.
\begin{proof}
Let $\alpha_1$ be the root of maximum modulus of the characteristic polynomial
of $(a_n)$. We prove that $c(\alpha_1)\not= 0$.

Under our assumptions $(e_n)$ is a sequence of real numbers. By definition of $c(z)$, $c(\alpha_1)=0$ if and only if
$e_n=0$ for all $n$. This is equivalent to
$a_n=\tilde{a}_n$ for all $n$.
If $\tilde{a}_n=a_n$ for all $n$, then $(\tilde{a}_n)$
is an integer valued lrs.
By a result of Fatou (see e.g. \cite{Salem:63}), we know that the
formal power series
$$
\sum_{n=0}^{\infty} \tilde{a}_n x^n
$$
with integer coefficients represents a rational function $P(x)/Q(x)$
%%JHE: added "integer coefficients"
with $P,Q\in \Z[x]$ and $Q(0)=1$.
Putting $Q(x)=\sum_{i=0}^m q_i x^i$ with $q_0=1$, the sequence
$(\tilde{a}_n)$ satisfies a linear recurrence:
$$
\tilde{a}_{n+m}+ q_1 \tilde{a}_{n+m-1} + \dots +q_m \tilde{a}_n=0
$$
for a sufficiently large $n$, as well as (\ref{lrek}), i.e.,
$$
\tilde{a}_{n+d} +A_{d-1}\tilde{a}_{n+d-1} + \dots + A_0\tilde{a}_n =0.
$$
Considering the characteristic polynomials
of these two recursions,
we have $Q(1/\alpha_1)=0$ and hence
$\alpha_1$ is an algebraic integer.
This is a contradiction and we know
that $c(\alpha_1)\neq 0$. From Corollary \ref{Diff}, we get the result.
\end{proof}

Corollary \ref{Diff2} has the following immediate consequence.
\begin{cor}
\label{Diff3}
If the characteristic polynomial of the nlrs $(s_n)$ from \eqref{rekursion}
has a single root of maximum modulus, and this is real, greater than one and not an algebraic integer, then
$\lim_{n\to\infty} |s_n-\tilde{s}_n|=\infty$.
\end{cor}

\section{Common values} \label{common}

Common values of lrs with algebraic terms are quite well investigated. Thanks to the theory of $S$-unit equations, developed by Evertse \cite{Evertse} and by van der Poorten and Schlickewei \cite{Sch_P}, M. Laurent \cite{Laurent} characterized those pairs of lrs's $(a_n)$, $(b_n)$ for which there are infinitely many pairs
of indices $(k,m)$ with $a_k=b_m$. His result is not effective.
A particular case of Laurent's result is that if $(a_n)$, $(b_n)$ have
separable characteristic polynomials then the set of $(k,m)$ with $a_k=b_m$
is either finite, or the union of a finite set and of finitely many
rational lines. A rational line is a set of the type $\{ (k,m)\in\Zz^2:\, k,m\ge K_0,\, Ak+Bm+C=0\}$,
where $K_0$ is a constant $\ge 0$ and $A,B,C$ are rational numbers.

We recall that two non-zero complex numbers $\alpha,\beta$ are multiplicatively
dependent if there are integers $m,n$, not both zero, with $\alpha^m\beta^n=1$,
and multiplicatively independent otherwise.
We say that a root $\alpha$ of a polynomial $P(x)$ with complex coefficients is
\emph{dominating} if $|\alpha |>|\beta |$ for every other root $\beta$ of $P$.

In the case that the characteristic polynomials
of the lrs $(a_n)$, $(b_n)$
both have a dominating root and if these two roots are
multiplicatively independent, Mignotte \cite{Mignotte:78}
proved that there are only finitely many $k,m$ with $a_k=b_m$
and gave an effective upper bound for them.
His result was generalized to sequences with at most three, not necessarily dominating roots by Mignotte, Shorey and Tijdeman \cite{MST}. One finds a good overview on effective results concerning common values of lnr's in the book of Shorey and Tijdeman \cite{ST}. In the above mentioned results the Binet formula \eqref{Binet} plays a central role.

Theorem \ref{Binet_t} gives a Binet-type formula for nlrs's, which suggests to study common values of such sequences.
The next result implies that the situation
for nlrs is quite different from that of lrs.

\begin{theorem}\label{counterexample}
Let $\alpha ,\beta$ be two multiplicatively independent real numbers $>1$.
Then there exist nlrs $(a_n)$, $(b_n)$ with integer terms,
having characteristic polynomials
with dominating roots $\alpha ,\beta$, respectively, such that
there are infinitely many pairs of non-negative
integers $(k,m)$ with $a_k=b_m$. This set of pairs $(k,m)$ has finite intersection
with every rational line.
\end{theorem}

In the proof we need some lemmas.

\begin{lemma}\label{very-good-approximation}
Let $a,b$ be positive real numbers with $a/b\not\in\Qq$ and let $C>1$.
Then there exists $c\in\Rr$ such that the inequality
\[
|ak-bm-c|<C^{-(k+m)}
\]
has infinitely many solutions in non-negative integers $k,m$.
\end{lemma}

\begin{proof}
We construct an infinite sequence of triples $(k_n,m_n,\varepsilon_n)$
($n=1,2,\ldots$) such that $0<\varepsilon_n<1$, $k_n,m_n$ are positive
integers with $|ak_n-bm_n|<\varepsilon_n$ for all $n$, and
\[
\varepsilon_{n+1}<\min \big(\half \varepsilon_n, (2C)^{-(k_1+\cdots +k_n+m_1+\cdots +m_n)}\big).
\]
The existence of such an infinite sequence follows easily from Dirichlet's
approximation theorem or the continued fraction expansion of $a/b$. Now put
$s_n:=ak_n-bm_n$ and
\[
c:=\sum_{n=1}^{\infty} s_n.
\]
This series is easily seen to be convergent. Further we have, on putting
$k_n':=k_1+\cdots +k_n$, $m_n':=m_1+\cdots +m_n$,
\[
|ak_n'-bm_n'-c|\leq \sum_{l=n+1}^{\infty} |s_l|
<2(2C)^{-(k_n'+m_n')}<C^{-(k_n'+m_n')}.
\]
This clearly proves our lemma.
\end{proof}

\begin{lemma}\label{very-good-approximation-2}
Let $\alpha ,\beta$ be multiplicatively independent reals $>1$ and $C>1$.
Then there exists $\gamma >1$ such that the inequality
\[
|\alpha^k-\gamma\beta^m|<C^{-(k+m)}
\]
has infinitely many solutions in positive integers $k,m$.
\end{lemma}

\begin{proof}
By the previous lemma, there exist $\gamma>1$ and infinitely many pairs
of positive integers $(k,m)$, such that
\[
|k\log\alpha -m\log\beta -\log\gamma |<(2\beta C)^{-(k+m)}.
\]
Using the inequality $|e^x-1|\leq 2|x|$ for real $x$ sufficiently close to $0$,
we infer that there are infinitely many pairs $(k,m)$ of positive integers
such that
\begin{eqnarray*}
|\alpha^k-\gamma\beta^m|&=&\gamma\beta^m\cdot |\alpha^k\beta^{-m}\gamma^{-1}-1|
\\
&\le& 2\gamma\beta^m |k\log\alpha -m\log\beta-\log\gamma |
\\
&\le& 2\gamma\beta^m (2\beta C)^{-(k+m)}<C^{-(k+m)}.
\end{eqnarray*}
\end{proof}

\begin{proof}[Proof of Theorem \ref{counterexample}]
The previous lemma implies that there exists $\gamma >0$ such that
$[\alpha^k]-[\gamma\beta^m]\in\{ -1,0,1\}$ for infinitely many pairs
of non-negative integers $k,m$. This implies that there are
$u\in\{ -1,0,1\}$ and infinitely
many pairs of non-negative integers $k,m$ such that $[\alpha^k]-[\gamma\beta^m]=u$.
Now define $(a_n)$, $(b_n)$ by $a_n:=[\alpha^n]$, $b_n:=[\gamma\beta^n+u]$.
These are easily seen to be nlrs with dominating roots $\alpha$, $\beta$,
respectively, and clearly, $a_k=b_m$ for infinitely many pairs $k,m$.
There is $C>0$ such that $|k\log\alpha -m\log\beta |\leq C$
for all pairs of non-negative integers $k,m$ with $a_k=b_m$.
Since by assumption, $\log\alpha/\log\beta\not\in\Q$, only finitely many of these
pairs $(k,m)$ can lie on a given rational line.
This completes our proof.
\end{proof}

Below, we consider the set of pairs $(k,m)$ satisfying $a_k=b_m$ for two
given nlrs $(a_n)$, $(b_n)$ in more detail. One of our results is that if
$(a_n)$, $(b_n)$ satisfy the conditions of Theorem \ref{counterexample}
and if moreover $\alpha$, $\beta$ are algebraic, then the set of these pairs
$(k,m)$ is very sparse.
%
%Some attempts learned us that the situation is completely different here. The reason is that the roots of $P(x)$ may not be algebraic numbers.
%Even if $A_0,\dots, A_{d-1}$ are algebraic, i.e. the roots of $P(x)$ are algebraic,
%we have no reason to expect that $c(\alpha)$ is an algebraic number.
%It is likely that these values are transcendental.
%In such a case the subspace theorem and lower bounds for linear forms in
%logarithms of algebraic numbers,
%the most powerful tools of the theory of Diophantine equations,
%cannot be applied directly.
%Therefore we believe it is of some interest
%to study common values for some class of nlrs's.

The main ingredient of our proof is
an effective lower bound for linear forms in logarithms of algebraic numbers.
We use here a theorem of Matveev \cite{Matveev}. For our qualitative
result below it would be enough to use a less explicit form,
but we could save almost nothing with it.
Before formulating the theorem we have to define the {\it absolute logarithmic height} - $h(\beta)$ - of an algebraic number $\beta$. Let $\beta$ be an algebraic number of degree $n$ and denote by $b_0$ the leading coefficient of its defining polynomial.
Further, denote by $\beta = \beta^{(1)}, \dots, \beta^{(n)}$ the (algebraic) conjugates of $\beta$. Then
$$
h(\beta) = \frac{1}{n} \left(\log |b_0| + \sum_{j=1}^n \log \max\{|\beta^{(j)}|,1 \}\right).
$$

\begin{theorem} \label{Matveev}
Let $\gamma_1,\dots,\gamma_t$ be positive real algebraic numbers in a real algebraic number field $K$ of degree $D$ and $b_1,\dots, b_t$ rational integers such that
$$
\Lambda := \gamma_1^{b_1} \cdots \gamma_t^{b_t} -1\not= 0.
$$
Then
$$
|\Lambda| > \exp \left(-1.4 \times 30^{t+3} \times t^{4.5} \times D^2(1 + \log D)(1 + \log B )A_1 \cdots A_t \right),
$$
where
$$
B \ge \max \{|b_1|, \dots, |b_t|\},
$$
and
$$
A_i \ge \max \{ Dh(\gamma_i),|\log \gamma_i|, 0.16\}\ \
\mbox{for all}\  i = 1,\dots, t.
$$

\end{theorem}

Now we are in the position
to state and prove our main result of this section.

\begin{theorem} \label{main1}
Let $(a_n)$ and $(b_n)$ be two nlrs.
Assume that the characteristic polynomials of $(a_n)$, $(b_n)$ have
dominating roots $\alpha ,\beta$ respectively, and that $\alpha, \beta$
are real algebraic, have absolute value $>1$, and are
multiplicatively independent.

Then there exist effectively computable constants $K_0,K_1,K_2$ depending only on the characteristic polynomials, the initial values and the sizes of the error terms of $(a_n)$ and $(b_n)$ such that if $(k_1,m_1),(k_2,m_2) \in \Z^2$ are solutions of the diophantine equation
\begin{equation} \label{egyenlo}
a_k = b_m
\end{equation}
with $K_0\le k_1 < k_2$ then $k_2 > k_1 + K_1\exp(K_2 k_1)$.
\end{theorem}

\begin{proof}
As $|\alpha|, |\beta|>1$ there exist by Theorem \ref{Binet_t}
non-zero constants $\gamma ,\delta$
depending only on the starting terms of the sequences $(a_n)$, $(b_n)$
and the coefficients
of their characteristic polynomials,
and a constant $\varepsilon$ depending only on the second largest zeros such that
\begin{equation} \label{eq1}
a_k = \gamma\alpha^k + O(\alpha^{k(1-\varepsilon)}),
\quad b_m = \delta\beta^m + O(\beta^{m(1-\varepsilon)})
\end{equation}
hold for all large enough $k,m$.
Thus if equation \eqref{egyenlo} holds then we have
$$
\gamma\alpha^k - \delta\beta^m =
 O(|\alpha |^{k(1-\varepsilon)})+ O(|\beta |^{m(1-\varepsilon)}).
$$
For fixed $m$ this inequality has finitely many solutions in $k$. Let $K_0$ be a large enough constant, which we will specify later, and assume that \eqref{egyenlo} has at least one solution $(m,k)$ with $k>K_0$.

We may assume $\alpha,\beta >0$ without loss of generality. Indeed, otherwise we consider the cases $k,m$ odd and even separately. Moreover we may also assume $\alpha^k > \beta^m$ (equality cannot occur as $\alpha$ and $\beta$ are multiplicatively independent). Then the last inequality implies
$$
\left| \frac{\delta}{\gamma} \frac{\beta^m}{\alpha^k} -1 \right| < C_1 \alpha^{-k\varepsilon} + C_2 \beta^{-m\varepsilon}.
$$
Here and in the sequel the constants $C_1,C_2,\dots$
are effectively computable and %%added by JHE
depend on the parameters of the sequences, i.e. on their
initial terms and the heights of the coefficients and their characteristic
polynomials, on $\varepsilon$ and on the upper bound of the terms of the error sequences only.

We now assume that $K_0$ is large enough %%JHE: slightly modified
so that the right hand side of the last inequality is less than $1/2$. Then
$$
\beta^m > \left|\frac{\gamma}{2\delta}\right|\cdot  \alpha^k,
$$
thus
\begin{equation} \label{eq2}
\left| \frac{\delta}{\gamma} \frac{\beta^m}{\alpha^k} -1 \right| < C_3 \alpha^{-k\varepsilon}.
\end{equation}
This inequality seems to have already the form for which Matveev's theorem \ref{Matveev} could be applied. Unfortunately we are not yet so far because we know nothing about the arithmetic nature of $\gamma$ and $\delta$. They can be (and are probably usually) transcendental numbers. Thus Theorem \ref{Matveev} is not applicable and we cannot deduce an upper bound for $\max\{k,m\}$. On the other hand inequality \eqref{eq2} is strong enough to allow us to prove that the sequence of solutions of \eqref{egyenlo} is growing very fast.

Indeed, let $(k_1,m_1), (k_2,m_2)$ be solutions of \eqref{egyenlo} such that $K_0<k_1<k_2$. Then \eqref{eq2} holds for both solutions and we get
$$
\left| \frac{\delta}{\gamma} \frac{\beta^{m_1}}{\alpha^{k_1}} -
 \frac{\delta}{\gamma} \frac{\beta^{m_2}}{\alpha^{k_2}}\right| < 2 C_3 \alpha^{-k_1\varepsilon}.
$$
Dividing this inequality by the first term, which lies by \eqref{eq2} in the interval $(\half,\frac{3}{2})$ we get
\begin{equation} \label{eq3}
\left|\Lambda \right| < C_4 \alpha^{-k_1\varepsilon},
\end{equation}
where
$$
\Lambda = \beta^{m_2-m_1} \alpha^{k_1-k_2} -1.
$$
As $\alpha$ and $\beta$ are positive real numbers and multiplicatively independent we have
$\Lambda\neq 0$, thus we may apply Theorem \ref{Matveev} to it, with $t=2$.
In our situation, %%JHE: slightly modified
$D$ is the degree of the number field $\Q(\alpha,\beta)$, $A_1,A_2$ are constants depending only on the coefficients of the characteristic polynomials of the sequences.
Further $b_1 = m_2-m_1$ and $b_2=k_1-k_2$. We proved above that
if $k,m$ are integers with $a_k=b_m$ and $k>K_0$ then either
$$
\alpha^k> \beta^m >
\left|\frac{\gamma}{2\delta}\right| \alpha^k
$$
or
$$
\beta^m > \alpha^k> \left|\frac{\delta}{2\gamma}\right| \beta^m
$$
holds. We have in both cases
$$
\left| m - \frac{\log \alpha}{\log \beta}k \right| < C_5.
$$
This implies
\begin{equation} \label{eq5}
\left| |m_1 -m_2| - \frac{\log \alpha}{\log \beta}|k_1-k_2| \right| < C_6.
\end{equation}
Thus $|b_1|< C_7 |b_2| + C_6$ and Theorem \ref{Matveev} implies
$$
|\Lambda| > \exp(-C_8 D^2 (1+\log D)A_1A_2 (2+ \log C_7 + \log(k_2-k_1) ).
$$
Comparing this inequality with \eqref{eq3} we obtain
$$
C_9 \log(k_2-k_1) + C_{10} > C_{11} \varepsilon k_1 - \log C_4,
$$
which implies
\begin{equation} \label{eq4}
k_2 > k_1 + K_2 \exp(K_1k_1),
\end{equation}
with $K_1 = C_{11}\varepsilon/C_9$ and $K_2= \exp(-C_{10}/C_9 - (\log C_4)/C_9)$.
\end{proof}

We now consider the case that the $(a_n)$, $(b_n)$
have characteristic polynomials with multiplicatively dependent
dominant roots. We show that in this case, if the number of pairs
$(k,m)$ with $a_k=b_m$ is infinite, then apart from at most finitely many
exceptions they lie on a rational line.

\begin{theorem} \label{m-dependent}
Let $(a_n)$ and $(b_n)$ be nlrs's.
Assume that the characteristic polynomials of both sequences are separable,
and have dominating roots $\alpha ,\beta$ with $|\alpha |>1$, $|\beta |>1$
which are multiplicatively dependent.
If the equation
\begin{equation} \label{egy1}
a_k = b_m
\end{equation}
has infinitely many solutions in non-negative integers $k,m$ then there exist integers $u,v,w$ such that for all but finitely many solutions we have $k = um/v + w/v$.
\end{theorem}

\begin{proof}
Like in the proof of Theorem \ref{main1} we write
$$
a_k = \gamma\alpha^k + O(|\alpha |^{k(1-\varepsilon)}) \quad \mbox{and} \quad b_m = \delta\beta^m + O(|\beta |^{m(1-\varepsilon)}),
$$
with $\gamma ,\delta\not= 0$. As $\alpha$ and $\beta$ are multiplicatively dependent, there exist positive integers $u,v$ such that
$$
\alpha^u = \beta^v,
$$
i.e., there exists a $v$-th root of unity $\zeta$ with
$$
\beta = \zeta\alpha^{u/v}.
$$
If $a_k = b_m$ then
\begin{equation} \label{egy2}
\gamma\alpha^k - \zeta^m \delta\alpha^{um/v}
= O(|\alpha |^{k(1-\varepsilon)}) + O(|\alpha |^{um(1-\varepsilon)/v}).
\end{equation}
Assume that $k-um/v>\ell_1$, where the integer $\ell_1$ is so large that
$$
|\delta \alpha^{-\ell_1}| < \left|\frac{\gamma}{3}\right|.
$$
If, moreover, $k$ is large enough then dividing \eqref{egy2} by $\alpha^k$
would make the absolute value of the right hand side smaller than
$\left|\frac{\gamma}{3}\right|$ too, which is impossible. Thus if \eqref{egy1} has infinitely many solutions $k,m$ then
$k-um/v \le \ell_1$. Similarly, if $um/v -k > \ell_2$, where the integer $\ell_2$ is so large that
$$
|\gamma\alpha^{-\ell_2}| < \left|\frac{\delta}{3}\right|,
$$
then repeating the former argument we get again a contradiction.
Thus setting $\ell = \max\{\ell_1,\ell_2\}$ we must have $|k-um/v| \le \ell$ for all but finitely many solutions of \eqref{egy1}.

%{\bf Answer the last inquiry of the referee.}
Thus we have shown that for all but finitely many solutions $(k,m)$ of \eqref{egy1}
there is
$w\in [-v\ell,v\ell]\cap \Z$ such that $k-um/v=w/v$.
We have to show that $w$ is independent of the choice of $(k,m)$.
Clearly, there
is $w$ such that
$k-um/v=w/v$ holds for infinitely many solutions $(k,m)$ of \eqref{egy1}.
Dividing \eqref{egy2} by $|\alpha |^k$ we see that
$$
|\gamma - \zeta^m \delta\alpha^{w/v}|=O(|\alpha |^{-k\varepsilon})
$$
for these solutions $(k,m)$. Since the left-hand side of this inequality
assumes only finitely many values and the right-hand side can become arbitrarily
small, there must be an integer $r$ such that
\[
\gamma =\zeta^r\delta\alpha^{w/v}.
\]
We show that this uniquely determines $w$. Indeed, suppose we have
$\gamma =\zeta^{r_1}\delta\alpha^{w_1/v}=\zeta^{r_2}\delta\alpha^{w_2/v}$
for two tuples of integers $(r_1,w_1)$, $(r_2,w_2)$. Then $\alpha^{(w_1-w_2)/v}$
is a root of unity, which implies $w_1=w_2$ since by assumption, $\alpha$
is not a root of unity. This shows that $k=um/v+w/v$ holds for all but finitely
many solutions $(k,m)$ of \eqref{egy1}.
\end{proof}

A nearly immediate consequence of Theorem \ref{m-dependent} is the following assertion.

\begin{cor}
  Let $(a_n)$ be an nlrs. Assume that its characteristic polynomial is separable,
and has a dominant root $\alpha$ with $|\alpha |>1$.
Then the equation
\begin{equation} \label{egy11}
a_k = a_m
\end{equation}
has only finitely many solutions with $k\not= m$.
\end{cor}

\begin{proof}
We apply Theorem \ref{m-dependent}
in the situation that
the sequences under consideration are equal.
We have plainly $u=v=1$, i.e. $k=m+w$ holds with a fixed integer $w$ for all but finitely many solutions of \eqref{egy11}. Then
  $$
  \gamma\alpha^k(\alpha^{m-k} - 1) = O(|\alpha |^{k(1-\varepsilon)}),
  $$
which is absurd, if $m-k=w \not= 0$.
\end{proof}

\newpage

\begin{thebibliography}{99}

\bibitem{Akiyama-Borbely-Brunotte-Pethoe-Steiner:05}
S. Akiyama, T. Borb\'ely, H. Brunotte, A. Peth\H{o} and J. Thuswaldner, \emph{ Generalized radix representations and dynamical systems I}, Acta Math. Hungar., {\bf 108 (3)} (2005), 207--238.

\bibitem{Akiyama-Brunotte-Pethoe-Steiner:06}
S.~Akiyama, H.~Brunotte, A.~Peth\H{o}, and W.~Steiner, \emph{Remarks on a
  conjecture on certain integer sequences}, Periodica Math. Hungarica
  \textbf{52} (2006), 1--17.

\bibitem{Akiyama-Brunotte-Pethoe-Steiner:07}
\bysame, \emph{Periodicity of certain piecewise affine planar maps}, Tsukuba J.
  Math. \textbf{32} (2008), no.~1, 1--55.

\bibitem{Akiyama-Pethoe-:12}
S.~Akiyama and A.~Peth\H{o}, \emph{Discretized rotation has infinitely many periodic orbits}, Nonlinearity
 \textbf{26} (2013), 871--880.

\bibitem{Evertse}
J.H. Evertse, \emph{On sums of S-units and linear recurrences}, Compositio Math., {\bf 53} (1984), 225--244

\bibitem{HardyWright}
G.H.~Hardy and E.M.~Wright,
\emph{An introduction to the theory of numbers}, 4th. ed. (with corrections),
Oxford at the Clarendon Press, 1975.

\bibitem{Kirschenhofer-Pethoe-Thuswaldner:08}
P. Kirschenhofer, A. Peth\H o and J. Thuswaldner, \emph{On a family of three term nonlinear integer recurrences}, Int. J. Number Theory, {\bf 4} (2008), 135--146.

%\bibitem{Laurent}
%M. Laurent, \emph{Equations exponentielles polyn\^{o}mes et suites r\'{e}currentes %lin\'{e}aires}, Ast\'{e}risque, {\bf 147-148} (1987), 121 ??? 139.

\bibitem{Turan}
I.~L\'anczi and P.~Tur\'an, \emph{Sz\'amelm\'elet}, (Number Theory) in Hungarian, Tank\"onyvkiad\'o, Budapest, 1969.

\bibitem{Laurent}
M. Laurent, \emph{Equations exponentielles polyn\^{o}mes et suites r\'{e}currentes lin\'{e}aires II}, J. Number Theory {\bf 31} (1989), 24--53.

\bibitem{Lech}
C. Lech, \emph{A note on recurring series},
Ark. Math. {\bf 2} (1953), 417--421.

\bibitem{Lowenstein-Hatjispyros-Vivaldi}
J.H. Lowenstein, S.~Hatjispyros, and F.~Vivaldi, \emph{Quasi-periodicity,
  global stability and scaling in a model of Hamiltonian round-off}, Chaos
  \textbf{7} (1997), 49--56.

\bibitem{Matveev}
E.M. Matveev, \emph{An explicit lower bound for a homogeneous rational linear form in the logarithms of algebraic numbers, II}, Izv. Ross. Akad. Nauk Ser. Mat. \textbf{64} (2000), no. 6, 125--180; translation in Izv. Math. \textbf{64} (2000), no. 6, 1217--1269.

\bibitem{Mignotte:78}
M. Mignotte, \emph{Intersection des images de certaines suites r\'{e}currentes lin\'{e}aires}, Theor. Comp. Science {\bf 7} (1978), 117--122.

\bibitem{MST}
M. Mignotte, T.N. Shorey and R. Tijdeman, \emph{The distance between terms of an algebraic recurrence sequence}, J. Reine Angew. Math. {\bf 349} (1984), 63--76.

\bibitem{Petho:06}
A. Peth\H{o}, {\em Notes on $CNS$ polynomials and integral interpolation}, In: More Sets, Graphs and Numbers, Eds.: E. Gy\H{o}ry, G.O.H. Katona and L. Lov\'asz, Bolyai Soc. Math. Stud., 15, Springer, Berlin, 2006. pp. 301--315.

\bibitem{PV}
A. Peth\H{o} and P. Varga, {\em Canonical Number Systems over Imaginary Quadratic Euclidean Domains}, Coll. Math. to appear.

\bibitem{Sch_P}
A.J. van der Poorten and H.P. Schlickewei, \emph{The Growth Conditions for Recurrence Sequences}, Macquarie University, NSW, Australia (1982) Report 82.0041

\bibitem{Salem:63}
R. {Salem}, \emph{Algebraic numbers and {F}ourier analysis}, D. C. Heath and Co., Boston, Mass, 1963.

\bibitem{ST}
T.N. Shorey and R. Tijdeman, {\em Exponential diophantine equations}, Cambridge Tracts in Mathmematics, vol. 87, Cambridge University Press, 1986.
\end{thebibliography}
\end{document}